\documentclass[12pt]{article}
\usepackage{amssymb,amsthm,hyperref,amsmath,bm,amsfonts,tikz}
\usepackage{arydshln}
\usepackage{verbatim}
\usepackage{graphicx}
\usepackage{float}
\usepackage{setspace}
\usepackage{subfigure}
\usepackage{multirow}
\usepackage{subfigure}
\usepackage{color}
\usepackage{appendix}
\DeclareGraphicsExtensions{.pdf,.png,.jpg}
\usetikzlibrary{calc}
\newtheorem{theorem}{Theorem}[section]
\newtheorem{lemma}[theorem]{Lemma}

\newtheorem{definition}[theorem]{Definition}
\newtheorem{assumption}[theorem]{Assumption}

\newtheorem{remark}[theorem]{Remark}

\numberwithin{equation}{section}

\newcommand{\diag}{\mbox{diag}}
\newcommand{\J}{\boldsymbol{j}}
\newcommand{\E}{\boldsymbol{e}}

\title{Numerical boundary control of multi-dimensional discrete-velocity kinetic models}
\author{
Haitian Yang\thanks{E-mail: yht21@mails.tsinghua.edu.cn}\\
\small{\textit{Department of Mathematical Sciences, Tsinghua University, Beijing 100084, China.}}\\ \\
	Wen-An Yong\thanks{E-mail: wayong@tsinghua.edu.cn}\\
\small{\textit{Department of Mathematical Sciences, Tsinghua University, Beijing 100084, China.}}\\
\small{\textit{Beijing Institute of Mathematical Sciences and Applications, Beijing 101408, China.}}
}
\date{\today}

\begin{document}
	\maketitle{} 
\begin{abstract}
	This paper extends our recent results on multi-dimensional discrete-velocity models to the numerical level. By adopting an operator splitting scheme and introducing a suitable discrete Lyapunov function, we derive numerical control laws that ensure the corresponding numerical solutions decay exponentially in time. To handle the stiff source term, we also use an implicit scheme for the collision part and prove the stability of the resulting schemes. The theoretical results are validated through three  numerical simulations for the two-dimensional coplanar model.
\end{abstract}

   \hspace{-0.5cm}\textbf{Keywords: Discrete-velocity models, Lyapunov function, Numercial boundary control, Semi-implicit schemes.}
 \small{}\\
\section{Introduction}

This paper continues our recent work \cite{yang2025}. In \cite{yang2025}, we studied the boundary control problem for multi-dimensional (multi-D) discrete-velocity models, which are of first-order hyperbolic systems. In this paper, our goal is to develop suitable numerical schemes for these models and establish a numerical stabilization result. We begin by reviewing  existing results on continuous and numerical boundary control.

Over the past two decades, the boundary control problem of first-order hyperbolic systems has attracted much attention in the mathematical and engineering community due to its wide range of applications. Three main methods have been developed for addressing this problem: the characteristics method, the Lyapunov function method and the backstepping method.  In \cite{li2010controllability}, exact boundary controllability for quasi-linear systems was shown via the characteristics method. The Lyapunov function method with a smallness assumption on the source terms was presented in \cite{bastin2016stability}. 
In \cite{Hu2016,Hu2019}, general linear or quasi-linear coupled systems were treated with the backstepping method. For an overview of the Lyapunov function method and the backstepping method, we refer to the survey papers \cite{hayat2021boundary} and \cite{vazquez2024backstepping}, respectively. 
While most of the aforementioned works deal with the spatially one-dimensional problems, a few  efforts have been made to extend the Lyapunov analysis to multi-D systems recently. In \cite{herty2022stabilization, herty2024boundary}, the authors employed an exponential-type Lyapunov function under certain assumptions on the source terms. In our recent work \cite{yang2023feedback, yang2025}, we constructed  the Lyapunov functions based on a physically relevant dissipation structure \cite{yong1999singular} and the specific features of the coefficient matrices, which enabled us to achieve boundary stabilization for the 2-D Saint-Venant equations and multi-D discrete-velocity kinetic models.

In view of the application background of the control problems, it is meaningful to design appropriate numerical schemes for solving the equations and proving the corresponding  exponential stability. In the one-dimensional case, the pioneering work \cite{banda2013numerical} deals with the numercial stabilization of conservation laws. In \cite{gottlich2017numerical} and \cite{banda2020numerical}, the authors used the operator splitting scheme  to study numerical stabilization for hyperbolic balance laws around uniform and non-uniform steady states, respectively. A key assumption in the last   two works is that the source term is fully dissipative, see \cite[(3.14)]{gottlich2017numerical} and \cite[(16)]{banda2020numerical}. The methods can be applied to the semi-linear case \cite{Stephan2023}.  
In higher dimensions, the only relevant result we know is the very recent paper \cite{herty2024numerical}, where the authors consider the diagonal coefficient matrices without a source term, which allows them to decouple the system into several independent scalar problems.

In this paper, we aim to extend our recent results \cite{yang2025} on multi-D discrete-velocity models to a numerical level. Motivated by the works mentioned above, we  adopt an operator splitting scheme. The main difficulty lies in the fact that the source term is not fully dissipative but possesses a stability structure \cite{yong1999singular,yong2008interesting}, which requires a more delicate analysis of certain weighted coefficients in the Lyapunov functional. To the best of our knowledge, this is the first work that addresses numerical boundary control for multi-D hyperbolic systems with coupled source terms. 
It is noteworthy that our analysis can be applied to many physcially relevant systems which satisfy the structural stability condition \cite{yong1999singular,yong2008interesting}, such as the one discussed in \cite{HERTY201612}. Moreover, since the discrete-velocity models are used to simulate the Boltzmann equation, the source term is usually stiff. To handle this, we also use an implicit scheme for the collision part and prove that the corresponding schemes are stable. As far as we know, this is also the first work on semi-implicit schemes in the context of numerical control theory for the first-order hyperbolic equations.

This paper is organized as follows. In Section \ref{S2}, we review the properties of multi-D discrete-velocity models \cite{yang2025}. In Section \ref{S3}, we specify the numerical boundary conditions and introduce an operator splitting scheme that combines an upwind scheme for the advection term with a forward Euler scheme for the collision term. We then prove that the numerical schemes, together with suitable numerical boundary conditions, are exponentially stable. We also show that an implicit scheme for the collision part is stable, which can be used to handle the stiff source term in practical simulations for Boltzmann equations. In Section \ref{S4}, we consider the 2-D coplanar model \cite{platkowski1988discrete} with numerical experiments to illustrate our theoretical results. The concluding remarks are presented in the last section.

\section{Boundary control of multi-D discrete-velocity models} \label{S2}
In this section, we briefly recall the stabilization results in \cite{yang2025}. Discrete-velocity models describe the evolution of gas density by tracking particle populations moving with a finite number of prescribed fixed velocities (called discrete velocities). Depending on the number and directions of these velocities, different models have been proposed, such as the coplanar model, the Broadwell model, the Carleman model, and so on. This is why we refer to them in the plural as “models.” They are systems of multi-D first-order semi-linear hyperbolic equations with source terms (balance laws). They are computationally feasible approximations of the Boltzmann equation \cite{palczewski1997consistency} and can be used to simulate the behaviors of
rarefied gases \cite{inamuro1990numerical}, to investigate the shock structure \cite{gatignol1975kinetic}, to study the Couette and Rayleigh flow \cite{Gatignol1975}, etc. 
For a detailed overview of the history and applications of discrete-velocity models, we refer to the survey paper \cite{platkowski1988discrete}.

As shown in \cite[Section~II]{yang2025}, we linearized the semi-linear systems around a uniform steady state (a constant vector). The linearized system
\begin{equation} \label{2.1}
	f_t(t,x)+\sum_{i=1}^d \Lambda_i f_{x_i}(t,x)=Qf(t,x)
\end{equation}
is defined on $(t,x) \in [0,\infty)\times \Omega$, where $\Omega$ a bounded domain in $\mathbb R^d$ with Lipschitz continuous boundary. In practice, $\Omega$ typically represents a gas reaction container. $f = f(t,x) \in \mathbb{R}^K$ denotes the unknown state,
$\Lambda_i = \diag\{\lambda_{1i}, \dots, \lambda_{Ki}\} \in \mathbb{R}^{K\times K}$ $(i = 1, \dots, d)$ are constant diagonal matrices,
and $Q = (Q_{km})_{1 \le k,m \le K}$ is a constant $K \times K$ matrix. The physical meaning of the $k$-th ($k=1,\cdots,K$) component of $f$ is the fluctuation of the density of gas particles moving with the velocity $v_k=(\lambda_{k,1},\cdots,\lambda_{k,d})$, and the matrix $Q$ describes the effect of binary collisions between particles. Notice that the Boltzmann equation has infinitely
many velocities while the discrete-velocity models replace the  continuous
velocity variable with a finite number of velocities, therefore, this is the origin of their names: the velocity variable of the Boltzmann equation is discretized.

To solve the system \ref{2.1}, we need to prescribe the boundary conditions on $\partial \Omega$. At boundary point $x \in \partial \Omega$, we denote by $\mathbf{n}(x)=(n_1(x),\cdots,n_d(x))$  the unit outward normal vector. For each $x \in \partial \Omega$, the components of $f$ corresponding to the negative (resp. positive) entries of the diagonal matrix
$
\sum_{j=1}^d n_j(x)\Lambda_j
$
are referred to as incoming (resp. outgoing) variables at the boundary point, denoted by $f_-(t,x)$ (resp. $f_+(t,x)$). Here, the minus (resp. plus) sign indicates that the information propagates into (resp. out of) the domain against (resp. along) the outward normal direction.
According to the classic theory \cite{Serre2006,higdon1986initial,Majda1975InitialboundaryVP,Russell1978}, the proper boundary condition specifies the incoming variable
at each boundary point $x \in \partial \Omega$ :
\begin{equation} \label{2.2}
	f_-(t,x) \text{ in terms of } f_+(t,y), \quad y \in \partial \Omega,
\end{equation}
\noindent with  $y=x$ or not.
Unlike the one-dimensional case \cite[p. 243, (A.5)]{bastin2016stability}, it appears impossible to express general non-local boundary conditions in a compact form in the multi-D setting. This is because  boundary conditions may involve non-local couplings, leading to infinitely many possible choices. Therefore, we only use (\ref{2.2}) to express boundary conditions and illustrate it with examples (see  \cite{yang2025} and Section  \ref{S4} in this paper). A trivial example of (\ref{2.2}) is 
\begin{equation} \label{2.3}
	f_-(t,x)=0,\quad \forall x \in \partial \Omega.
\end{equation}

The discrete-velocity models satisfy the structural stability condition \cite{yong1999singular}, which corresponds to the celebrated Onsager reciprocal relations in non-equilibrium thermodynamics. Moreover, as discrete-velocity models approximate the Boltzmann equation, the structural stability condition can be viewed as an approximate form of the famous H-theorem \cite{platkowski1988discrete} for the Boltzmann equation. It also reflects the fact that the physical conservation law  holds universally, regardless of whether the underlying thermodynamic system is in equilibrium or not; see \cite{yong2008interesting}. With this structure, we pointed out in \cite[Lemma 1]{yang2025} that
\begin{lemma} \label{L1}
	There exists an invertible matrix $P$  and a diagonal positive definite matrix $\Lambda_0=\diag\{\lambda_{10},\cdots,\lambda_{K0}\}$ such that
	\begin{equation} \label{2.4}
		PQP^{-1}=-\begin{pmatrix}
			0 & 0 \\
			0 & \Lambda
		\end{pmatrix}
	\end{equation}
	and
	\begin{equation} \label{2.5}
		\qquad \Lambda_0 Q=-P^T\begin{pmatrix}
			0 & 0 \\
			0 & \Lambda
		\end{pmatrix} P,
	\end{equation}
	where $\Lambda \in \mathbb{R}^{r\times r}$ is a diagonal matrix with positive entries and $ r \leq K.$ 
\end{lemma}
\noindent Note that the fully dissipative systems studied in \cite{gottlich2017numerical,banda2020numerical}
correspond to $r=K$, while this work allows $r<K.$

For these discrete velocities $v_1,\cdots,v_N$,
we impose the following technical assumption:
\begin{assumption} \label{A1}
	Each velocity $v_k$  $(k=1,\cdots,K)$ is a non-zero vector.
\end{assumption}
\noindent Physically, it means that the gas system under consideration does not contain static particles. Mathematically, this assumption ensures that the diagonal matrix $\sum_{i=1}^d \Lambda_i^2$ is strictly positive definite.

In \cite{yang2025} based on Lemma \ref{L1} and Assumption \ref{A1}, we constructed the following Lyapunov function 
$$
L(f)(t)= \int_{\Omega}  f^T(t,x) \left[\alpha\Lambda_0+\exp\left(-\sum_{i=1}^d \Lambda_ix_i\right)\right] f(t,x) dx,
$$
where $\alpha$ is a positive parameter suitably chosen. With this Lyapunov function, we prove that the system (\ref{2.1}) is exponentially stable provided that the boundary conditions (\ref{2.2}) satisfy
\begin{equation} \label{2.6}
	\int_{\partial \Omega}  f^T(t,x) \left[\alpha\Lambda_0+\exp\left(-\sum_{i=1}^d \Lambda_ix_i\right)\right]\left( \sum_{i=1}^d n_i\Lambda_i\right) f(t,x) d\sigma \geq 0.
\end{equation}
Here the notation $d\sigma$ stands for the Lebesgue measure on $\partial \Omega.$

Since discrete-velocity models are originally introduced as computationally feasible approximations of the Boltzmann equation, it is 
practically meaningful to design a suitable numerical scheme for the discrete-velocity models whose correpsonding numerical solution  is also exponentially stable, which will be presented in the next section.

\section{Main results} \label{S3}

\subsection{Numerical schemes}
For simplicity, we consider the cubic domain $\Omega = (0,1)^d$, and  discretize each spatial direction uniformly. Given a spatial length $\Delta x=1/N$ with $N$ a positive integer,  the grid points of the interval $[0,1]$ of the $k$-th direction are chosen as
$$
x_{k,j} = j\Delta x, \quad j = 0, \dots, N.
$$ 
We introduce the multi-index set 
$$
\mathcal{J} = \Bigl\{ \J = (j_1,\dots,j_d) \;\Big|\; 
j_i \in \{1,\dots,N-1\}, \; i=1,\dots,d \Bigr\}.
$$
and call the points $x_{\J}:=(x_{1,j_1},\cdots,x_{d,j_d})$ for $\J \in \mathcal{J}$ the interior points of the domain $\Omega$. These points lie inside the domain, and we aim to compute the numerical solution there. We also introduce boundary points $x_{\J}$ for $\J$ in the multi-index set
$$
\partial \mathcal{J} = \Bigl\{ \J = (j_1,\dots,j_d) \;\Big|\; 
\exists\, i \ \text{such that } j_i = 0 \ \text{or} \ j_i = N \Bigr\}.
$$
These points lie on $\partial \Omega$ and are used to impose our boundary conditions. 
See Figure \ref{fig:2d_spatial_discretization} for an illustration of the case $N=4, d=2.$
\begin{figure}[htbp]
	\centering
	\begin{tikzpicture}[scale=4, >=stealth, every node/.style={font=\scriptsize}]
		
		\begin{scope}
			\draw[->] (-0.1,0) -- (1.1,0) node[right] {$x_1$};
			\draw[->] (0,-0.1) -- (0,1.1) node[above] {$x_2$};
		\end{scope}

		\node[red,left] at (-0.25,0.5) {};
		\node[red,below] at (0.5,-0.25) {};
		
		\draw[thick] (0,0) rectangle (1,1);

		\begin{scope}[shift={(1.1,1.01)}]
			\draw[black, thick] (0,0) rectangle (0.8,0.3);
			
			\filldraw[red] (0.10,0.22) circle (0.6pt);
			\node[right] at (0.15,0.22) {Boundary points};
			
			\filldraw[blue] (0.10,0.08) circle (0.6pt);
			\node[right] at (0.15,0.08) {Interior points};
		\end{scope}
		
		\foreach \x in {0,0.25,0.5,0.75,1} {
			\filldraw[red] (\x,0) circle (0.6pt);   
			\filldraw[red] (\x,1) circle (0.6pt);   
		}
		\foreach \y in {0,0.25,0.5,0.75,1} {
			\filldraw[red] (0,\y) circle (0.6pt);   
			\filldraw[red] (1,\y) circle (0.6pt);   
		}
		
		\foreach \x in {0.25,0.5,0.75} {
			\foreach \y in {0.25,0.5,0.75} {
				\filldraw[blue] (\x,\y) circle (0.6pt);
			}
		}
		
		\foreach \x in {0.25,0.5,0.75} {
			\draw[dashed,gray] (\x,0) -- (\x,1); 
		}
		\foreach \y in {0.25,0.5,0.75} {
			\draw[dashed,gray] (0,\y) -- (1,\y); 
		}
		
	\end{tikzpicture}
	\caption{Spatial discretization of the square domain $\Omega = (0,1)^2$ with $N=4$.}
	\label{fig:2d_spatial_discretization}
\end{figure}
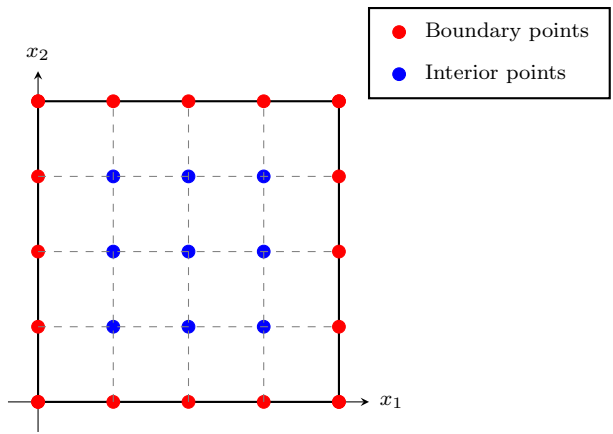

Next, we denote the time step by $\Delta t$ and the discrete time points as
$$
t_n = n \Delta t, \quad n \in \mathbb{N}.
$$
We introduce the notation $
f_{\boldsymbol{j}}^n $ with the multi-index $\J=(j_1,\cdots,j_d)$, to denote the numerical approximation of the solution $f(n\Delta t,x_{1,j_1}, \dots, x_{d,j_d})$ of the system (\ref{2.1}).
Note that $f_{\J}^n$ is a $K$-dimensional vector, we denote its $k$-th component by $f_{k,\J}^n$. 
Due to the large number of introduced symbols, we summarize their meanings in Figure \ref{fig:f_jkn_notation}. \\

\begin{figure}[htbp]
	\centering
	\begin{tikzpicture}[>=stealth, every node/.style={font=\scriptsize}]
		
		\node (f) at (0,0) {\Large $f_{k,\textcolor{blue}{\J}}^{\textcolor{red}{n}}$};
		
		\draw[dashed, thick, gray, rounded corners]
		($(f.north west)+(-0.1,0.1)$) rectangle ($(f.south east)+(0.1,-0.1)$);
		
		\draw[->, thick,dashed] (-0.32,0.2) -- (-1.2,1) node[above left,align=center,font=\scriptsize] {Numerical approximation of \\ $k$-th component of $f(t_n, x_{1,j_1}, \dots, x_{d,j_d})$};
		\draw[->, thick, red] (0.2,0.4) -- (1.2,1) node[above right, align=center,font=\scriptsize, red] {The time step $n$};
		\draw[->, thick] (-0.15,-0.3) -- (-1.2,-1) node[below left, align=center,font=\scriptsize] {The $k$-th component};
		\draw[->, thick, blue] (0.35,-0.3) -- (1.2,-1) node[below right,align=center, font=\scriptsize, blue] {The multi-index $\J = (j_1, \dots, j_d)$ \\ denotes the spatial grid point \\ $(x_{1,j_1},\cdots,x_{d,j_d})$};
		
	\end{tikzpicture}
	\caption{Illustration of the notation $f_{\J,k}^n$.}
	\label{fig:f_jkn_notation}
\end{figure}
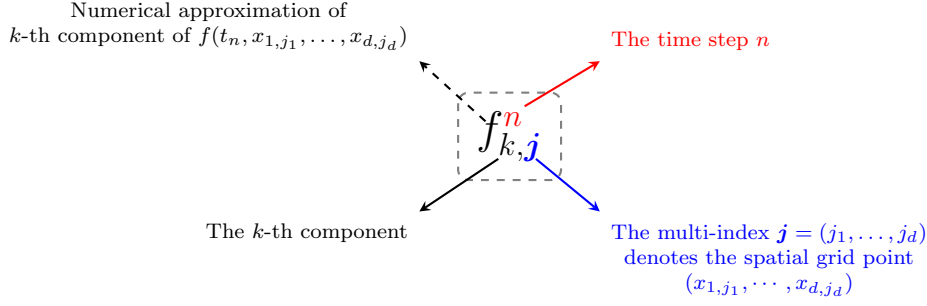

We use the following operator splitting schemes to numerically solve (\ref{2.1}). The idea is to divide the balance law into two parts which 
can be solved independently. For the advection
part $f_t + \sum_{i=1}^d \Lambda_i f_{x_i} = 0$, an upwind scheme is applied for discretizing the first-order operator. For the collision part
$f_t = Qf$, the forward Euler scheme is used. The splitting schemes are standard in the Lattice Boltzmann Method \cite{junk2009weighted}, and also have been used in the numerical control problem in the one-dimensional case \cite{gottlich2017numerical}.

Recall that the $i$-th ($i=1,\cdots,d$) coefficient matrix $\Lambda_i$ is $\diag\{\lambda_{1i},\cdots,\lambda_{Ki}\}$ and the source term matrix $Q$ is $(Q_{km})_{1\leq k,m\leq K},$ therefore, for $n \in \mathbb N$ and $\J \in \mathcal{J}$, the numerical schemes are formulated as
\begin{align}
	&\tilde{f}_{k,\J}^{n}=f_{k,\J}^{n}-\Delta t\sum_{i:\lambda_{ki}>0}\lambda_{ki}\frac{f_{k,\J}^{n}-f_{k,\J-\E_i}^{n}}{\Delta x} -\Delta t\sum_{i:\lambda_{ki}<0}\lambda_{ki}\frac{f_{k,\J+\E_i}^{n}-f_{k,\J}^{n}}{\Delta x}  , \label{3.1}\\
	&f_{k,\J}^{n+1}=\tilde{f}_{k,\J}^{n}+\Delta t\sum_{m = 1}^K Q_{km}\tilde{f}_{m,\J}^{n},\quad k=1,\cdots,K;\quad n\in \mathbb N;\quad \J \in \mathcal{J}.\label{3.2}
\end{align}	
Here and below, the notation $\E_i$ represents the $d$-tuple whose $i$-th component is $1$ and  others are all $0.$ For a given $k \in \{1,\cdots,K\}$, the notation $\sum\limits_{i:\lambda_{ki}>0}$ (resp. $\sum\limits_{i:\lambda_{ki}<0}$ ) denotes the summation over all directions $i=1,\ldots,d$ such that $\lambda_{ki}>0$ (resp. $\lambda_{ki}<0$).

To initialize the numerical schemes, we need a numerical initial value. Given an initial condition $f_0(x) \in L^2(\Omega)$ for system (\ref{2.1}), we take the numerical initial condition $f^0_{\J}$ as the  cell average of $f_0(x)$ with the cell centered at $x_{\J}.$

Note that the scheme (\ref{3.1}) involves the values at boundary points. It is necessary to prescribe the numerical boundary conditions.  
At boundary point $x_{\J}$ ($\J \in \partial \mathcal{J}$) and time $n\Delta t$,
we denote the numerical incoming variables as $f_{-,\J}^n$,  
where the minus sign indicates the components corresponding to the incoming variables, i.e. the component of $f_-(n\Delta t,x_{\J})$. On the other hand, we take the numerical outgoing variables as the numerical solution at the interior grid point adjacent to the boundary point.
For instance, the numerical outgoing variables at the boundary point $(1,x_{2,j_2},\cdots,x_{d,j_d})$ and time $n\Delta t$ are chosen as $f^n_{+,(N-1,j_2,\cdots,j_d)}$, where the plus sign indicates the components corresponding to the outgoing variables, i.e. the component of $f_+(n\Delta t,1,x_{2,j_2},\cdots,x_{d,j_d})$.   

\begin{remark} \label{R1}
	The upwind scheme (\ref{3.1}) could be defined up to the boundary points, since for $\lambda_{ki}>0 $ (resp. $\lambda_{ki}<0$) the scheme remains valid even when $\J_i = N$ (resp. $\J_i=0$). Nevertheless, for simplicity of notations and presentation, we restrict $\J$ to the interior index set $\mathcal{J}$. This choice introduces a slight asymmetry in the numerical boundary value: the numerical outgoing variables are not taken as the discrete values of the continuous boundary variables, whereas the incoming variables are. 
\end{remark}

With the numerical incoming and outgoing variables prescribed above, the numerical boundary conditions are taken as the discretized version of (\ref{2.2}). Again, there are infinitely many choices and we cannot write them compactly.
A trivial example is to choose \begin{equation} \label{3.3} 
	f^n_{-,\J}=0,\quad \forall \J\in \partial \mathcal{J},
\end{equation} 
which is the disretization version of (\ref{2.3}).

\begin{remark}
	Numerical boundary control refers to the selection of suitable numerical boundary conditions to ensure the stability of the schemes, where the parameters governing the dependence between numerical incoming and outgoing  variables serve as the control variables. The trivial numerical boundary condition (\ref{3.3}) corresponds to setting all these parameters to zero. Several non-trivial examples will be presented in Section~\ref{S4}.
\end{remark}


\subsection{Numerical boundary stabilization}

We denote by $f^n := \{f^n_{\J}\}_{\J \in \mathcal{J}}$ the collection of all discrete interior values at time step $n$, which is the numerical solution to (\ref{2.1}) at the time $n\Delta t$. We define the following $\ell^2$-norm:
$$
\| f^n \|_{\ell^2} 
:= \left( \sum_{\J \in \mathcal{J}} 
(f_{\J}^n)^{T} f_{\J}^n \; (\Delta x)^d 
\right)^{1/2}.
$$

With this norm, the numerical exponential stability is defined as follows:
\begin{definition} \label{D1}
	The numerical schemes (\ref{3.1})-(\ref{3.2}) are called exponentially stable in the sense of $\ell^2$-norm, if there exist  positive constants $C$ and $\nu$, independent of $\Delta x$ and $\Delta t$, such that,  for any initial value $f^0$, the corresponding numerical solution satisfies
	$$
	\|f^n \|_{\ell^2}\leq Ce^{-\nu (n\Delta t)} \| f^0\|_{\ell^2}.
	$$
\end{definition}

Motivated by the proof in the continuous case \cite{yang2025}, we introduce the corresponding numerical Lyapunov function $L(\cdot)$:
$$
\begin{aligned}
	L(f^n)=&(\Delta x)^d\sum_{\J\in \mathcal{J}} (f_{\J}^{n})^T\left(\alpha \Lambda_0+\exp\left(-\sum_{l=1}^d \Lambda_l x_{l,j_l}\right)\right) f_{\J}^{n} \\
	=&(\Delta x)^d\sum_{\J\in \mathcal{J}} \sum_{k=1}^K (f_{k,\J}^{n})^2 \left(\alpha \lambda_{k0}+\exp\left(-\sum_{l=1}^d \lambda_{kl}x_{l,j_l}\right)\right),
\end{aligned}
$$
with $\alpha$ a positive constant to be determined. Here, we have used the  
notations $\Lambda_i=\diag\{\lambda_{1i},\cdots,\lambda_{Ki}\}$ for $i=0,1,\cdots,d.$

Denote the maximum (resp. minimum) of the entries of the strictly positive diagonal matrices $\exp\left(-\sum_{l=1}^d \Lambda_l x_l\right)$ (over $\bar{\Omega}$) and  $\Lambda_0$ as $M,\lambda_M$ (resp. $m,\lambda_m$). We first choose $\alpha \geq M/\lambda_M$, consequently
\begin{equation} \label{3.4}
	\alpha\lambda_m  \| f^n\|_{\ell^2}^2 \leq (m+\alpha \lambda_m) \| f^n\|_{\ell^2}^2 \leq L(f^n)\leq (M+\alpha \lambda_M)\| f^n\|_{\ell^2}^2 \leq 2\lambda_M \alpha \| f^n\|_{\ell^2}^2.
\end{equation}

With the above preparations, we now state the first lemma, which deals with the advection part (\ref{3.1}) of the scheme. We exploit Assumption \ref{A1} to establish interior numerical damping and introduce suitable boundary control laws to achieve stabilization.
\begin{lemma} \label{L2}
	Assuming that the time step $\Delta t$ satisfies
	\begin{equation} \label{3.5}
		\begin{aligned}
			\sum_{i=1}^d \frac{\Delta t}{\Delta x} |\lambda_{ki}|\leq 1,\quad k=1,\cdots,K,
		\end{aligned}
	\end{equation}
	and the numerical boundary conditions ensure $\mathcal{B}\leq  0$ with $\mathcal{B}$ defined below, 
	then for the numerical scheme (\ref{3.1}), we have
	$$
	\frac{L(\tilde f^n)-L(f^n)}{\Delta t}\leq  -\frac{m\mu}{2\lambda_M\alpha} L(f^n)
	$$
	where $\mu$ is a positive number independent of $\Delta x$ and $\Delta t.$
	
\end{lemma}

\begin{proof}
	For each $k=1,\cdots,K$, we rewrite the advection scheme (\ref{3.1}) as
	$$
	\tilde{f}_{k,\J}^{n}=\left(1-\sum_{i:\lambda_{ki}>0}\frac{\Delta t}{\Delta x} \lambda_{ki}+\sum_{i:\lambda_{ki}<0}\frac{\Delta t}{\Delta x} \lambda_{ki}\right)f_{k,\J}^{n}+\sum_{i:\lambda_{ki}>0}\frac{\Delta t}{\Delta x} \lambda_{ki}f_{k,\J-\E_i}^{n}+\sum_{i:\lambda_{ki}<0}\frac{\Delta t}{\Delta x} (-\lambda_{ki})f_{k,\J+\E_i}^{n}.
	$$

	Thanks to (\ref{3.5}), all the coefficients in the right-hand of the above equality are between $0$ and $1$, hence
	by Jensen's inequality, we have
	$$
	(\tilde{f}_{k,\J}^{n})^2\leq \left(1-\sum_{i:\lambda_{ki}>0}\frac{\Delta t}{\Delta x} \lambda_{ki}+\sum_{i:\lambda_{ki}<0}\frac{\Delta t}{\Delta x} \lambda_{ki}\right)(f_{k,\J}^{n})^2+\sum_{i:\lambda_{ki}>0}\frac{\Delta t}{\Delta x} \lambda_{ki}(f_{k,\J-\E_i}^{n})^2+\sum_{i:\lambda_{ki}<0}\frac{\Delta t}{\Delta x} (-\lambda_{ki})(f_{k,\J+\E_i}^{n})^2.
	$$
	
	By the definitions of $L(\tilde{f}^n)$ and $L(f^n)$ and the last inequality, we have
	$$
	\begin{aligned}
		\frac{L(\tilde f^n)-L(f^n)}{\Delta t}=&(\Delta x)^d\sum_{\J\in \mathcal{J}} \sum_{k=1}^K \frac{(\tilde f_{k,\J}^{n})^2-(f_{k,\J}^{n})^2}{\Delta t} \left(\alpha\lambda_{k0}+\exp\left(-\sum_{l=1}^d \lambda_{kl}x_{l,j_l}\right)\right)\\
		\leq & (\Delta x)^d \sum_{k=1}^K 
		\sum_{i:\lambda_{ki}>0}\sum_{\J\in \mathcal{J}}\frac{\lambda_{ki}}{\Delta x} \left( (f_{k,\J-\E_i}^{n})^2-(f_{k,\J}^{n})^2 \right)\left(\alpha\lambda_{k0}+\exp\left(-\sum_{l=1}^d \lambda_{kl}x_{l,j_l}\right)\right) \\
		&+(\Delta x)^d \sum_{k=1}^K\sum_{i:\lambda_{ki}<0}\sum_{\J\in \mathcal{J}}\frac{\lambda_{ki}}{\Delta x} \left((f_{k,\J}^{n})^2-(f_{k,\J+\E_i}^{n})^2\right) \left(\alpha\lambda_{k0}+\exp\left(-\sum_{l=1}^d \lambda_{kl}x_{l,j_l}\right)\right)
		\\
		:=&I+II. 
	\end{aligned}
	$$
	
	For $I$, we shift each index $\J-\E_i$ to $\J$, which is analogous to the integration by parts in the continuous case. Similarly, we shift each index $\J+\E_i$ to $\J$ for $II.$ Consequently, we have
	$$
	\begin{aligned}
		\frac{L(\tilde f^n)-L(f^n)}{\Delta t}=&(\Delta x)^d\sum_{k=1}^K  
		\sum_{i:\lambda_{ki}>0} \sum_{\J\in \mathcal{J}} \frac{\exp( -\lambda_{ki}\Delta x)-1}{\Delta x} \lambda_{ki} (f_{k,\J}^{n})^2 \left(\exp\left(-\sum_{l=1}^d \lambda_{kl}x_{l,j_l}\right)\right) 
		\\
		&+(\Delta x)^d \sum_{k=1}^K  
		\sum_{i:\lambda_{ki}<0} \sum_{\J\in \mathcal{J}}\frac{1-\exp( \lambda_{ki}\Delta x)}{\Delta x} \lambda_{ki} (f_{k,\J}^{n})^2 \left(\exp\left(-\sum_{l=1}^d \lambda_{kl}x_{l,j_l}\right)\right) \\
		&+(\Delta x)^{d-1} \sum_{k=1}^K \sum_{i:\lambda_{ki}>0}\sum_{\substack{j_i=0}} \lambda_{ki}(f_{k,\J}^n)^2\left(\alpha\lambda_{k0}+\exp\left(-\sum_{l=1}^d \lambda_{kl}x_{l,j_l}\right)\exp(-\lambda_{ki}\Delta x)\right) \\
		&-(\Delta x)^{d-1} \sum_{k=1}^K \sum_{i:\lambda_{ki}>0} \sum_{\substack{j_i=N-1}} \lambda_{ki}(f_{k,\J}^n)^2\left(\alpha\lambda_{k0}+\exp\left(-\sum_{l=1}^d \lambda_{kl}x_{l,j_l}\right)\right)\\
		&+(\Delta x)^{d-1} \sum_{k=1}^K \sum_{i:\lambda_{ki}<0}\sum_{\substack{j_i=1}} \lambda_{ki}(f_{k,\J}^n)^2\left(\alpha\lambda_{k0}+\exp\left(-\sum_{l=1}^d \lambda_{kl}x_{l,j_l}\right)\right) \\
		&-(\Delta x)^{d-1} \sum_{k=1}^K \sum_{i:\lambda_{ki}<0} \sum_{\substack{j_i=N}} \lambda_{ki}(f_{k,\J}^n)^2\left(\alpha\lambda_{k0}+\exp\left(-\sum_{l=1}^d \lambda_{kl}x_{l,j_l}\right)\exp(\lambda_{ki}\Delta x)\right)\\
		:=&\mathcal{I}+\mathcal{B}.
	\end{aligned}
	$$
	Here, $\mathcal{I}$ denotes the first two terms on the right-hand side of the last equation, representing the interior terms, while $\mathcal{B}$ refers to the remaining four boundary terms. Note that $\mathcal{B}$ tends to (\ref{2.6}) formally when $\Delta x \to 0.$
	Given a direction $i$, the notation $\sum\limits_{j_i=0}$ indicates that the summation runs over all the multi-index $\J$ whose $i$-th component $j_i=0$, and other compoents are in $\{1,\cdots,N-1\}.$ The other three notations $\sum\limits_{j_i=N-1},\sum\limits_{j_i=1},\sum\limits_{j_i=N}$ are defined similarly.  \\
	
	\textbf{For the interior terms $\mathcal{I}$:} Note that for all $\Delta x \in (0,1]$, we have
	\begin{equation*} 
		\begin{aligned}
			\frac{\exp( -\lambda_{ki}\Delta x)-1}{\Delta x} \lambda_{ki}&<0,\quad \text{for}\ \lambda_{ki}>0,\\ 
			\frac{1-\exp( \lambda_{ki}\Delta x)}{\Delta x} \lambda_{ki}&<0,\quad \text{for}\  \lambda_{ki}<0.
		\end{aligned}
	\end{equation*}
	Thanks to Assumption \ref{A1}, for each $k=1,\cdots,K$, there exists at least one direction $i_0$ such that $\lambda_{ki_0}\neq 0$. Hence we have
	$$
	\mathcal{I}= -(\Delta x)^d\sum_{k=1}^K  \sum_{\J\in \mathcal{J}}
	\tilde \mu_k (f_{k,\J}^{n})^2 \left(\exp(-\sum_{l=1}^d \lambda_{kl}x_{l,\J_l})\right)
	$$
	where
	$$
	\tilde \mu_k=\tilde \mu_k(\Delta x)=-\left[\sum_{i:\lambda_{ki}>0} \frac{\exp( -\lambda_{ki}\Delta x)-1}{\Delta x} \lambda_{ki}+ \sum_{i:\lambda_{ki}<0} \frac{1-\exp( \lambda_{ki}\Delta x)}{\Delta x} \lambda_{ki}\right]>0
	$$
	for each $k=1,\cdots,K$. 
	Note that
	$$
	\lim_{\Delta x \to 0+}\tilde \mu_k(\Delta x) =\sum_{i=1}^{d} \lambda_{ki}^2>0
	$$
	and $\Delta x\leq 1$,
	consequently we can choose 
	$$
	\mu_k=\min_{\Delta x\leq 1} \tilde \mu_k(\Delta x)>0
	$$
	which is  independent of $\Delta x$.
	Let $\mu=\min\limits_{k=1,\cdots,K}\{\mu_1,\cdots,\mu_k\}>0$ 
	and recall that the minimum of  the strictly positive continuous  function $\exp(-\sum_{i=1}^d \Lambda_i x_i)$ on $\bar{\Omega}$ is $m$, hence 
	$$
	\mathcal{I}\leq -m\mu\|f^n\|_{\ell^2}^2 \leq -\frac{m\mu}{2\lambda_M\alpha}L(f^n),
	$$
	where we have used (\ref{3.4}) in the last inequality sign.\\

	\textbf{For the boundary terms $\mathcal{B}$}: 
	The contributions of incoming variables to  the boundary term  $\mathcal{B}$,
	$$
	\begin{aligned}
		&(\Delta x)^{d-1} \sum_{k=1}^K \sum_{i:\lambda_{ki}>0} \sum_{\substack{j_i=0}} \lambda_{ki}(f_{k,\J}^n)^2\left(\alpha\lambda_{k0}+\exp(-\sum_{l=1}^d \lambda_{kl}x_{l,j_l})\exp(-\lambda_{ki}\Delta x)\right) \\
		&-(\Delta x)^{d-1} \sum_{k=1}^K \sum_{i:\lambda_{ki}<0} \sum_{\substack{j_i=N}} \lambda_{ki}(f_{k,\J}^n)^2\left(\alpha\lambda_{k0}+\exp(-\sum_{l=1}^d \lambda_{kl}x_{l,j_l})\exp(\lambda_{ki}\Delta x)\right)\\
	\end{aligned}
	$$
	are always non-negative.
	On the other hand, for  $\lambda_{ki} < 0$ and $j_i = 1$, or $\lambda_{ki} > 0$ and $j_i = N-1$, the term $f_{k,\J}^n$ corresponds to the outgoing variables, and the contributions to $\mathcal{B}$
	$$
	\begin{aligned}
		&-(\Delta x)^{d-1} \sum_{k=1}^K \sum_{i:\lambda_{ki}>0} \sum_{\substack{j_i=N-1}} \lambda_{ki}(f_{k,\J}^n)^2\left(\alpha\lambda_{k0}+\exp(-\sum_{l=1}^d \lambda_{kl}x_{l,j_l})\right)\\
		&+(\Delta x)^{d-1} \sum_{k=1}^K \sum_{i:\lambda_{ki}<0} \sum_{\substack{j_i=1}} \lambda_{ki}(f_{k,\J}^n)^2\left(\alpha\lambda_{k0}+\exp(-\sum_{l=1}^d \lambda_{kl}x_{l,j_l})\right) \\
	\end{aligned}
	$$
	are always non-positive. Consequently, $\mathcal{B}\leq 0$ holds, at least when the trivial numerical boundary condition (\ref{3.3}) is imposed.
	Several non-trivial numerical boundary conditions will be presented in Section \ref{S4}.
\end{proof}

\begin{remark} \label{R2}
	Formally, a smaller value of $\mathcal{B}$ corresponds to a stronger dissipative effect of the numerical scheme. Note that
	\begin{equation*} 
		\begin{aligned} \mathcal{B} \geq &-(\Delta x)^{d-1} \sum_{k=1}^K \sum_{i:\lambda_{ki}>0} \sum_{\substack{j_i=N-1}} \lambda_{ki}(f_{k,\J}^n)^2\left(\alpha\lambda_{k0}+\exp(-\sum_{l=1}^d \lambda_{kl}x_{l,j_l})\right)\\ &+(\Delta x)^{d-1} \sum_{k=1}^K \sum_{i:\lambda_{ki}<0} \sum_{\substack{j_i=1}} \lambda_{ki}(f_{k,\J}^n)^2\left(\alpha\lambda_{k0}+\exp(-\sum_{l=1}^d \lambda_{kl}x_{l,j_l})\right), \end{aligned} 
	\end{equation*}
	where the equality holds if and only if the trivial numerical boundary condition (\ref{3.3}) is imposed, corresponding to the most dissipative case. This phenomenon will be illustrated in Simulation II, Section \ref{S4}.
\end{remark}

We now turn to the collision scheme (\ref{3.2}). The analysis relies on the structural stability condition presented in Lemma \ref{L1}, and  does not involve any boundary term. Moreover, rewriting (\ref{3.2}) into matrix form will facilitate the stability analysis. To this end, for each $\J=(j_1,\cdots,j_d)\in \mathcal{J}$, we denote $f_{\J}^n=(f_{1,\J}^n,\cdots,f_{K,\J}^n)^T$. Consequently, the numerical scheme (\ref{3.2}) becomes
$$
f^{n+1}_{\J}=\tilde f^n_{\J}+\Delta tQ\tilde f^n_{\J}.
$$
We also use the notations
$
\Lambda_{\J}=\diag\left(\exp\left(-\sum_{i=1}^d \lambda_{i1}x_{i,j_i}\right),\cdots,\exp\left(-\sum_{i=1}^d \lambda_{iK}x_{i,j_i}\right)\right)
$ and  
\begin{equation*}
	\begin{pmatrix}
		\tilde u_{\J}^n \\
		\tilde q_{\J}^n
	\end{pmatrix}=P\tilde f_{\J}^n
\end{equation*}
with the same partition in Lemma \ref{L1}, that is, $\tilde u_{\J}^n \in \mathbb{R}^{K-r},\tilde q_{\J}^n \in \mathbb{R}^{r}.$

\begin{lemma} \label{L3}
	For any $\epsilon>0,$ we have
	$$
	\frac{L(f^{n+1})-L(\tilde f^n)}{\Delta t}\leq \epsilon \|\tilde u^n\|^2_{\ell^2}+(\frac{C^2_1}{4\epsilon}+C_2-\alpha \lambda)\|\tilde q^n\|^2_{\ell^2}+\Delta t \tilde M L(\tilde f^n),
	$$
	where the positive constants $C_1,C_2,\tilde M$ are  independent of $\Delta x$ and $\Delta t.$ 
\end{lemma}

\begin{proof}
	By the definition of $L(f^{n+1})$ and $L(f^n)$, and applying the numerical scheme (\ref{3.2}), we have
	$$
	\begin{aligned}
		\frac{L(f^{n+1})-L(\tilde f^n)}{\Delta t}=&\frac{(\Delta x)^d}{\Delta t} \sum_{\J \in \mathcal{J}} \sum_{k=1}^N \Bigl( (f^{n+1}_{k,\J})^2-(\tilde f^{n}_{k,\J})^2\Bigr)\left(\alpha \lambda_{k0}+\exp\left(-\sum_{l=1}^d \lambda_{kl}x_{l,j_l}\right)\right)\\
		=&\frac{(\Delta x)^d}{\Delta t}\sum_{\J \in \mathcal{J}} (f_{\J}^{n+1})^T (\alpha \Lambda_0+\Lambda_{\J})f_{\J}^{n+1}-\frac{(\Delta x)^d}{\Delta t}\sum_{\J \in \mathcal{J}} (\tilde f_{\J}^{n})^T (\alpha \Lambda_0+\Lambda_{\J})\tilde f_{\J}^{n} \\
		=&\frac{(\Delta x)^d}{\Delta t}\sum_{\J \in \mathcal{J}} (\tilde f^n_{\J}+\Delta tQ\tilde f^n_{\J})^T (\alpha \Lambda_0+\Lambda_{\J})	(\tilde f^n_{\J}+\Delta tQ\tilde f^n_{\J})\\&-\frac{(\Delta x)^d}{\Delta t}\sum_{\J \in \mathcal{J}} (\tilde f_{\J}^{n})^T (\alpha \Lambda_0+\Lambda_{\J})\tilde f_{\J}^{n} \\
		=&2\sum_{\J \in \mathcal{J}} \tilde f^n_{\J}Q^T(\alpha \Lambda_0+\Lambda_{\J})\tilde f^n_{\J}(\Delta x)^d+\Delta t \sum_{\J \in \mathcal{J}} (Q\tilde f^n_{\J})^T(\alpha \Lambda_0+\Lambda_{\J})(Q\tilde f^n_{\J})(\Delta x)^d\\
		:=&\mathcal{L}+\mathcal{H}.
	\end{aligned}
	$$
	Here $\mathcal{L}$ and $\mathcal{H}$ refer to the lower and higher-order term of $\Delta t$, respectively.
	
	\textbf{For $\mathcal{H}$}, we have
	$$
	\begin{aligned}
		\mathcal{H}=&\Delta t \sum_{\J \in \mathcal{J}}(Q\tilde{f}_{\J}^{n})^T(\Lambda_{\J}+\alpha \Lambda_0)(Q\tilde{f}_{\J}^{n}) (\Delta x)^d  \\
		\leq& \Delta t(M+\lambda_M \alpha)\sum_{\J\in \mathcal{J}}(Q\tilde{f}_{\J}^{n})^T(Q\tilde{f}_{\J}^{n}) (\Delta x)^d\\
		\leq &\Delta t(M+\lambda_M \alpha)\| Q\|^2\| \tilde{f}^{n}\|_{\ell^2}^2 \\
		\leq & \Delta t\frac{M+\lambda_M \alpha }{\lambda_m\alpha} \| Q\|^2L(\tilde 
		f^n),
	\end{aligned}
	$$
	where $M,\lambda_M$ are introduced in (\ref{3.4}). Here and below, $\|Q \|$ denotes the 2-norm (i.e. the largest singular value) of a matrix $Q$. Recall that $\alpha \geq M/\lambda_M$, then
	$$
	\frac{M+\lambda_M \alpha }{\lambda_m\alpha} \| Q\|^2 \leq \frac{2\lambda_M\|Q\|^2}{\lambda_m}:=\tilde M.
	$$
	Consequently, we have $\mathcal{H}\leq \Delta t \tilde M L(\tilde f^n).$

	\textbf{For $\mathcal{L}$}, we have 
	$$
	\begin{aligned}
		\mathcal{L}=&
		2\sum_{\J \in \mathcal{J}} (\tilde{f}_{\J}^n)^T Q^T(\Lambda_{\J}+\alpha \Lambda_0)(\tilde{f}_{\J}^n)(\Delta x)^d \\
		=& 2\sum_{\J\in \mathcal{J}} (P\tilde{f}_{\J}^n)^T (\begin{pmatrix}
			0 & 0\\
			0 & -\Lambda
		\end{pmatrix}  P^{-T}\Lambda_{\J}P^{-1}-\alpha \begin{pmatrix}
			0 & 0\\
			0 & \Lambda
		\end{pmatrix} )(P\tilde{f}_{\J}^n)(\Delta x)^d \\
		\leq & 2\sum_{\J\in \mathcal{J}} (P\tilde{f}_{\J}^n)^T \begin{pmatrix}
			0 & 0\\
			0 & -\Lambda
		\end{pmatrix}  P^{-T}\Lambda_{\J}P^{-1}(P\tilde{f}_{\J}^n)(\Delta x)^d-2\sum_{\J\in \mathcal{J}} \alpha \lambda (\tilde {q}_{\J}^n)^T(\tilde q_{\J}^n)(\Delta x)^d,
	\end{aligned}
	$$
	where we use Lemma \ref{L1} in the second equality sign, and we denote $\lambda>0$  the smallest eigenvalue of the diagonal positive definite matrix $\Lambda.$

	We rewrite $P^{-T}\Lambda_{\J}P^{-1}$ into sub-matrix corresponding to the partition of $P\tilde{f}^n_{\J}$:
	$$
	P^{-T}\Lambda_{\J}P^{-1}=\begin{pmatrix}
		\Lambda_{\J}^{11} & \Lambda_{\J}^{12} \\
		\Lambda_{\J}^{21} & \Lambda_{\J}^{22}
	\end{pmatrix},
	$$
	where $ \Lambda_{\J}^{22} \in \mathbb{R}^{r\times r}.$ Then we have
	$$
	\begin{pmatrix}
		0 & 0\\
		0 & -\Lambda
	\end{pmatrix}  P^{-T}\Lambda_{\J}P^{-1}=\begin{pmatrix}
		0 & 0\\
		-\Lambda\Lambda_{\J}^{21} & -\Lambda\Lambda_{\J}^{22}
	\end{pmatrix}.
	$$
	A key observation is that the left upper sub-matrix of the last matrix is zero. Hence we have
	$$
	\begin{aligned}
		2(P\tilde{f}_{\J}^n)^T \begin{pmatrix}
			0 & 0\\
			0 & -\Lambda
		\end{pmatrix}  P^{-T}\Lambda_{\J}P^{-1}(P\tilde{f}_{\J}^n)=&-2(\tilde u^n_{\J})^T\Lambda\Lambda_{\J}^{21}\tilde q^n_{\J}-2(\tilde q^n_{\J})^T\Lambda\Lambda_{\J}^{22}\tilde q^n_{\J} \\
		\leq& C_1|\tilde u^n_{\J}| |\tilde q^n_{\J}|+C_2|\tilde q^n_{\J}|^2,
	\end{aligned}
	$$
	where $|\tilde u^n_{\J}|=[(\tilde u^n_{\J})^T\tilde u^n_{\J}]^{1/2}$ denotes the usual Euclidean 2-norm of the vector $\tilde u^n_{\J}$, and
	$$
	C_1=2\max_{\J\in \mathcal{J}} (\| \Lambda \Lambda_{\J}^{21}\|),\quad  C_2=2\max_{\J\in \mathcal{J}} (\|\Lambda \Lambda_{\J}^{22}\|)
	$$
	which are positive constants independent of $\Delta x.$
	Consequently, thanks to the Cauchy-Schwarz  inequality, for any $\epsilon>0$, we have 
	$$
	\begin{aligned}
		\mathcal{L}\leq \sum_{\J\in \mathcal{J}} \left(C_1|\tilde u^n_{\J}| |\tilde q^n_{\J}|+C_2|\tilde q^n_{\J}|^2-\alpha \lambda|\tilde q^n_{\J}|^2 \right)(\Delta x)^d
		\leq& \sum_{\J\in \mathcal{J}} \left(\epsilon |\tilde u^n_{\J}|^2+(\frac{C^2_1}{4\epsilon}+C_2-\alpha \lambda)|\tilde q^n_{\J}|^2\right)(\Delta x)^d \\ =&\epsilon \|\tilde u^n\|^2_{\ell^2}+(\frac{C^2_1}{4\epsilon}+C_2-\alpha \lambda)\|\tilde q^n\|^2_{\ell^2}.
	\end{aligned}
	$$
	
	With the estimate of $\mathcal{H}$ and $\mathcal{L}$, we conclude this lemma.
\end{proof}

With the two lemmas proved above, we can state and prove our numerical stability result. 
\begin{theorem} \label{T1}
	Assume that $\Delta t$ satisfies (\ref{3.5}) and
	$$
	\Delta t\leq \frac{m\mu}{8\tilde M\lambda_M \alpha },
	$$
	where $\alpha$ is defined in (\ref{3.6}) below. 
	Then the numerical solution to the schemes (\ref{3.1}) and (\ref{3.2}) is exponentially stable in the sense of $\ell^2$-norm, provided that the numerical boundary conditions ensure $\mathcal{B} \leq 0$.
\end{theorem}

\begin{proof}
	Thanks to Lemma \ref{L2}, we have
	$$
	L(\tilde f^{n})\leq L(f^n).
	$$
	It follows that
	$$
	\| \tilde u^n\|_{\ell^2}^2 \leq \|  P \tilde f^n\|_{\ell^2}^2 \leq \| P\|^2\|   \tilde f^n\|_{\ell^2}^2 \leq \frac{\| P\|^2}{\alpha \lambda_m} L(\tilde f^n)\leq \frac{\| P\|^2}{\alpha \lambda_m} L(f^n).
	$$
	By the definition of $L(f^{n+1})$ and $L(f^n)$, and according to Lemma \ref{L2} and \ref{L3}, we have
	$$
	\begin{aligned}
		\frac{L(f^{n+1})-L(f^n)}{\Delta t}&=\frac{L( f^{n+1})-L(\tilde f^n)}{\Delta t}+\frac{L( \tilde f^{n})-L( f^n)}{\Delta t} \\ &\leq \epsilon \|\tilde u^n\|_{\ell^2}^2+(\frac{C^2_1}{4\epsilon}+C_2-\alpha \lambda)\|\tilde q^n\|_{\ell^2}^2+\Delta t \tilde M L(\tilde f^n) -\frac{m\mu}{2\lambda_M\alpha} L( f^n) \\
		&\leq   (-\frac{m\mu}{2\lambda_M\alpha} +\frac{\|P\|^2\epsilon}{\alpha \lambda_m}+\Delta t \tilde M)L(f^n)+(\frac{C^2_1}{4\epsilon}+C_2-\alpha \lambda)\|\tilde q^n\|_{\ell^2}^2.
	\end{aligned}
	$$
	We will choose the positive constants $\alpha$ and $\epsilon$ such that
	$$
	\begin{aligned}
		-\frac{m\mu}{2\lambda_M\alpha} +\frac{\|P\|^2\epsilon}{\alpha \lambda_m}+\Delta t \tilde M&<0,\\
		\frac{C^2_1}{4\epsilon}+C_2-\alpha \lambda& \leq 0.
	\end{aligned}
	$$
	Fix $\alpha$, first we choose 
	$$
	\epsilon=\frac{m\mu \lambda_m}{8\| P\|^2\lambda_M}, \quad \Delta t \leq \frac{m\mu}{8\tilde M\lambda_M \alpha },
	$$
	where $\epsilon$ is independent of $\alpha.$ Therefore, 
	$$
	-\frac{m\mu}{2\lambda_M\alpha} +\frac{\|P\|^2\epsilon}{\alpha \lambda_m}+\Delta t \tilde M \leq -\frac{m\mu}{4\lambda_M \alpha}<0.
	$$
	Then we can choose 
	\begin{equation} \label{3.6}
		\alpha =\max\ \{ \frac{C_1^2}{4\epsilon \lambda}+\frac{C_2}{\lambda},\frac{M}{\lambda_M} \} 
		=\max \ \{ \frac{2C_1^2\lambda_M\| P\|^2}{m \mu \lambda_m \lambda}+\frac{C_2}{\lambda},\frac{M}{\lambda_M} \}
	\end{equation}
	which is independent of $\Delta x$ and $\Delta t.$ Recall that the condition $\alpha \ge M / \lambda_M$ ensures the validity of (\ref{3.4}).
	
	Denote the positive constant $\frac{m\mu}{4\lambda_M \alpha}$ by $\mu_1$, and we have
	\begin{equation} \label{3.7}
		\frac{L(f^{n+1})-L(f^n)}{\Delta t}\leq -\mu_1L(f^n).
	\end{equation}
	Recursively applying (\ref{3.7}), we have 
	$$
	L(f^n)\leq (1-\mu_1\Delta t)^n L(f^0) \leq \exp(-\mu_1 (n\Delta t))L(f^0).
	$$
	Finally, by (\ref{3.4}), we have
	$$
	\| f^n\|^2_{\ell^2} \leq \frac{1}{\alpha \lambda_m}L(f^n) \leq \frac{1}{\alpha \lambda_m}\exp(-\mu_1 (n\Delta t))L(f^0) \leq \frac{2\lambda_M}{\lambda_m}\exp(-\mu_1 (n\Delta t))\| f^0\|_{\ell^2}^2.
	$$
	Clearly, the constants in Definition \ref{D1} are $C=\sqrt{2\lambda_M/\lambda_m}$ and $\nu=\mu_1/2$ .
\end{proof}

\subsection{Semi-implicit numercial schemes}
In simulations of the Boltzmann equation, the source term $Q$ in (\ref{2.1}) is often replaced by $Q/\sigma$, where the small positive constant $\sigma$ is proportional to the mean free path of the particles under consideration \cite{platkowski1988discrete}.   As a result, the source term becomes stiff. According to Theorem \ref{T1}, the time step $\Delta t$ decreases as $\sigma$ decreases (note that $\tilde M=\frac{2\lambda_M\|Q\|^2}{\sigma^2\lambda_m}$ in this case), which leads to unacceptably small time steps. To overcome this difficulty, we employ the following implicit scheme for the collision part: 
\begin{equation} \label{3.8} 
	f_{k,\J}^{n+1}=\tilde{f}_{k,\J}^{n}+\frac{\Delta t}{\sigma}\sum_{m = 1}^K Q_{km}{f}_{m,\J}^{n+1}
\end{equation}
or equivalently in its matrix form
\begin{equation} \label{3.9}
	(I_K-\frac{\Delta t}{\sigma} Q)f_{\J}^{n+1}=\tilde{f}_{\J}^{n},
\end{equation}
where $I_K$ is the $K\times K$-identity matrix. 

For the semi-implicit schemes (\ref{3.1}) and (\ref{3.8}), we have the following result, where the time step $\Delta t$ is independent of the source term.
\begin{theorem}
	Assume that $\Delta t$ satisfies (\ref{3.5}),
	then the numerical solution to the semi-implicit schemes (\ref{3.1}) and (\ref{3.8}) is exponentially stable in the sense of $\ell^2$-norm, provided that the numerical boundary conditions ensure $\mathcal{B} \leq 0$.
\end{theorem}

\begin{proof}
	Since (\ref{3.5}) holds and $\mathcal{B} \le 0$, the estimate in Lemma~\ref{L2}  for the scheme (\ref{3.1}) 
	remains valid. Therefore, it suffices to establish a counterpart of Lemma \ref{3.6} for the scheme (\ref{3.8}), which will be done in the following three steps.

	\textbf{Step 1. Estimating $\| u^{n+1}\|_{\ell^2}$.}  We multiply (\ref{3.9}) on the left by $(\Delta x)^d(f_{\J}^{n+1})^{T}\Lambda_{0}$ and then sum over all $\J \in \mathcal{J}$:
	\begin{equation*}
		(\Delta x)^d\sum_{\J \in \mathcal{J}} (f_{\J}^{n+1})^T\Lambda_0(I_K-\frac{\Delta t}{\sigma} Q)f_{\J}^{n+1}=	(\Delta x)^d\sum_{\J \in \mathcal{J}}(f_{\J}^{n+1})^T \Lambda_0\tilde{f}_{\J}^{n}.
	\end{equation*}
	Note that $\Lambda_0$ is positive definite, hence for all $\J \in \mathcal{J}$,
	$$
	(f_{\J}^{n+1})^T\Lambda_0\tilde{f}_{\J}^{n} \leq \frac{1}{2}(f_{\J}^{n+1})^T \Lambda_0f_{\J}^{n+1}+\frac{1}{2}(\tilde{f}_{\J}^{n})^T \Lambda_0\tilde{f}_{\J}^{n}
	$$ 
	by the weighted Cauchy-Schwarz inequality. Besides, thanks to (\ref{2.5}), we have
	$$
	(f_{\J}^{n+1})^T\Lambda_0Qf_{\J}^{n+1}=-(q_{\J}^{n+1})^T\Lambda q_{\J}^{n+1} \leq -\lambda(q_{\J}^{n+1})^Tq_{\J}^{n+1},
	$$
	where $$
	\begin{pmatrix}
		u_{\J}^{n+1} \\
		q_{\J}^{n+1}
	\end{pmatrix}=P f_{\J}^{n+1}
	$$  with $u_{\J}^{n+1} \in \mathbb R^{K-r},q_{\J}^{n+1} \in \mathbb R^{r}$ and $\lambda>0$ is  the smallest eigenvalue of the diagonal positive definite matrix $\Lambda$ as defined in the proof of Lemma \ref{L3}.
	Consequently, we have
	$$
	(\Delta x)^d	\sum_{\J \in \mathcal{J}} (f_{\J}^{n+1})^T\Lambda_0f_{\J}^{n+1}+(\Delta x)^d	\sum_{\J \in \mathcal{J}}\frac{\lambda\Delta t}{\sigma} (q_{\J}^{n+1})^Tq_{\J}^{n+1} \leq (\Delta x)^d\sum_{\J \in \mathcal{J}} (\tilde f_{\J}^{n})^T\Lambda_0\tilde f_{\J}^{n},
	$$
	in particular
	$$
	(\Delta x)^d\sum_{\J \in \mathcal{J}} (f_{\J}^{n+1})^T\Lambda_0f_{\J}^{n+1} \leq (\Delta x)^d\sum_{\J \in \mathcal{J}} (\tilde f_{\J}^{n})^T\Lambda_0\tilde f_{\J}^{n}.
	$$
	Let $C_3=\lambda_M/\lambda_m$ denote the ratio of the largest to the smallest eigenvalue of $\Lambda_0$. Then we have
	$$
	\| f^{n+1}\|^2_{\ell^2} \leq C_3 \| \tilde f^{n}\|^2_{\ell^2},
	$$
	and
	$$
	\begin{aligned}
		\| u^{n+1}\|^2_{\ell^2} \leq \| Pf^{n+1}\|_{\ell^2}^2 \leq C_3 \| P\|^2 \|\tilde  f^{n}\|^2_{\ell^2} \leq \frac{C_3 \| P\|^2}{\alpha \lambda_m} L(\tilde f^n) \leq \frac{C_3 \| P\|^2}{\alpha \lambda_m} L(f^n),
	\end{aligned}
	$$
	where  Lemma \ref{L2} has been used for the last inequality sign.

	\textbf{Step 2.  Estimating  $L(f^{n+1})$.} We multiply (\ref{3.9}) on the left by $(\Delta x)^d(f_{\J}^{n+1})^T(\alpha \Lambda_0+\Lambda_{\J})$, and sum over all $\J \in \mathcal{J}$:
	\begin{equation} \label{3.10}
		(\Delta x)^d	\sum_{\J\in \mathcal{J}} (f_{\J}^{n+1})^T(\alpha \Lambda_0+\Lambda_{\J})	(I_K-\frac{\Delta t}{\sigma} Q)f_{\J}^{n+1}=	(\Delta x)^d	\sum_{\J\in \mathcal{J}}(f_{\J}^{n+1})^T(\alpha \Lambda_0+\Lambda_{\J})\tilde{f}_{\J}^{n}.
	\end{equation}
	Note that $\alpha \Lambda_0+\Lambda_{\J}$ is positive definite, hence 
	$$
	(f_{\J}^{n+1})^T(\alpha \Lambda_0+\Lambda_{\J})\tilde{f}_{\J}^{n} \leq \frac{1}{2}(f_{\J}^{n+1})^T(\alpha \Lambda_0+\Lambda_{\J})f_{\J}^{n+1}+\frac{1}{2}(\tilde{f}_{\J}^{n})^T(\alpha \Lambda_0+\Lambda_{\J})\tilde{f}_{\J}^{n}.
	$$ 
	Consequently, (\ref{3.10}) becomes
	$$
	L(f^{n+1}) \leq L(\tilde f^n)+\frac{2\Delta t}{\sigma}	\sum_{\J\in \mathcal{J}} (f_{\J}^{n+1})^T(\alpha \Lambda_0+\Lambda_{\J})Qf_{\J}^{n+1}(\Delta x)^d.
	$$
	By following the same argument as in the proof of the lower-order term $\mathcal{L}$ in Lemma~\ref{L3}, we obtain 
	$$
	2	\sum_{\J\in \mathcal{J}} (f_{\J}^{n+1})^T(\alpha \Lambda_0+\Lambda_{\J})Qf_{\J}^{n+1}(\Delta x)^d \leq \epsilon \| u^{n+1}\|^2_{\ell^2}+(\frac{C^2_1}{4\epsilon}+C_2-\alpha \lambda)\| q^{n+1}\|^2_{\ell^2}
	$$
	for any $\epsilon>0.$
	
	Consequently, we have
	$$
	\frac{L(f^{n+1})-L(\tilde f^n)}{\Delta t} \leq  \frac{\epsilon}{\sigma} \| u^{n+1}\|^2_{\ell^2}+\frac{1}{\sigma}(\frac{C^2_1}{4\epsilon}+C_2-\alpha \lambda)\| q^{n+1}\|^2_{\ell^2}.
	$$
	
	\textbf{Step 3. Choosing $\alpha$ and completing the proof.} Now we have
	$$
	\begin{aligned}
		\frac{L(f^{n+1})-L(f^n)}{\Delta t}  &=\frac{L( f^{n+1})-L(\tilde f^n)}{\Delta t}+\frac{L( \tilde f^{n})-L( f^n)}{\Delta t} \\ &\leq \frac{\epsilon}{\sigma} \| u^{n+1}\|^2_{\ell^2}+\frac{1}{\sigma}(\frac{C^2_1}{4\epsilon}+C_2-\alpha \lambda)\| q^{n+1}\|^2_{\ell^2}-\frac{m\mu}{2\lambda_M\alpha} L( f^n) \\
		&\leq   (-\frac{m\mu}{2\lambda_M\alpha} +\frac{C_3\|P\|^2\epsilon}{\alpha \lambda_m \sigma})L(f^n)+\frac{1}{\sigma}(\frac{C^2_1}{4\epsilon}+C_2-\alpha \lambda)\|q^{n+1}\|_{\ell^2}^2,
	\end{aligned}
	$$
	where the estimate of $\| u^{n+1}\|_{\ell^2}^2$ is used in the last inequality sign.
	
	As in the proof of Theorem \ref{T1}, first we choose
	$$
	\epsilon=\frac{m\mu\lambda_m }{4\| P\|^2\lambda_M C_3}\sigma,
	$$ 
	then we choose
	\begin{equation} \label{3.11}
		\alpha =\max (\frac{C_1^2\lambda_MC_3\|P\|^2 }{m\mu \lambda \lambda_m}\frac{1}{\sigma}+\frac{C_2}{\lambda},\frac{M}{\lambda_M}),
	\end{equation}
	which leads to
	$$
	\frac{L(f^{n+1})-L(f^n)}{\Delta t} \leq -\frac{m\mu}{4\lambda_M\alpha}L(f^n).
	$$
	By following the same argument as in the proof of Theorem~\ref{T1}, the semi-implicit schemes can be showed to be  exponentially stable, where the constants in Definition~\ref{D1} are $C=\sqrt{2\lambda_M/\lambda_m}$ and $\nu=\frac{m\mu}{8\lambda_M \alpha}.$ This completes the proof.
\end{proof}

\begin{remark} \label{R3}
	Note that as the positive number $\sigma \to 0$, $\alpha$ tends to infinity (see (\ref{3.11})) then $\nu$ approaches zero. Hence, although the time step $\Delta t$ are independent of the source term $Q/\sigma$ when 
	implicit scheme (\ref{3.8}) is applied, the cost is a slower convergence rate for smaller values of $\sigma$. In contrast, for the explicit scheme (\ref{3.2}), $\alpha$ will not depend on $\sigma$ (see (\ref{3.6})); so the convergence rate remains fixed, but the time step is strictly limited. This highlights the trade-off between time steps and convergence rate, which will be illustrated in Simulation~III in Section~\ref{S4}.
\end{remark}

\section{Numerical simulations for the coplanar model} \label{S4}
In the previous section, we have shown that, the numerical boundary stabilization can be achieved for discrete-velocity models defined on  $\Omega=(0,1)^d$, at least under the trivial numerical boundary conditions (\ref{3.3}). Here, we carry out numerical simulations for the 2-D coplanar model \cite{cabannes1976etude,Gatignol1975,platkowski1988discrete}, which  belongs to the discrete-velocity models and can be applied to the study of the stationary plane flow around a wedge. We will also provide further examples of feasible numerical boundary conditions and compare their performance in numerical simulations.

Consider the gas confined in a square container $\Omega=(0,1)^2$ and postulate that the gas particles move with one of the four velocities of equal positive modulus $U$:
$$
u_1=(U,0),\ \ u_2=(-U,0),\ \ u_3=(0,U),\  \ u_4=(0,-U).
$$
Denote the number density functions by $\tilde{f}=(\tilde{f}_1,\tilde{f}_2,\tilde{f}_3,\tilde{f}_4)^T.$ Here $\tilde{f}_i=\tilde{f}_i(t,x_1,x_2)>0$  for $(t,x_1,x_2) \in [0,\infty) \times \bar{\Omega}$ is corresponding to the velocity $u_i$ ($i=1,\cdots,4$). The governing equation for each $\tilde{f}_k$ is
\begin{equation} \label{4.1}
	\begin{aligned}	\partial_t\tilde{f}_1+U\partial_{x_1}\tilde{f}_1&=\frac{1}{\sigma}\left(\tilde{f}_3\tilde{f}_4-\tilde{f}_1\tilde{f}_2\right),\\
		\partial_t\tilde{f}_2-U\partial_{x_1}\tilde{f}_2&=\frac{1}{\sigma}\left(\tilde{f}_3\tilde{f}_4-\tilde{f}_1\tilde{f}_2\right),\\
		\partial_t\tilde{f}_3+U\partial_{x_2}\tilde{f}_3&=-\frac{1}{\sigma}\left(\tilde{f}_3\tilde{f}_4-\tilde{f}_1\tilde{f}_2\right),\\
		\partial_t\tilde{f}_4-U\partial_{x_2}\tilde{f}_4&=-\frac{1}{\sigma}\left(\tilde{f}_3\tilde{f}_4-\tilde{f}_1\tilde{f}_2\right),
	\end{aligned}
\end{equation}
where $\sigma>0$  is proportional to the mean free path of the particles.

For a uniform steady state $f_e=(f_{1}^e,f_{2}^e,f_{3}^e,f_{4}^e)^T$ with positive components, we denote the fluctuation $f:=\tilde{f}-f_e=(f_1,f_2,f_3,f_4)^T.$ Then we linearize (\ref{4.1}) at this uniform steady state :
\begin{equation} \label{4.2}
	\begin{aligned}
		\partial_t f+\Lambda_1 \partial_{x_1} f+\Lambda_{2} \partial_{x_2}  f= \frac{1}{\sigma}Qf,
	\end{aligned}
\end{equation}
with $\Lambda_1=\diag(U,-U,0,0), \Lambda_2=\diag(0,0,U,-U)$,
$$
Q=\begin{pmatrix}
	-f_{2}^e   & -f_{1}^e  & f_{4}^e  & f_{3}^e\\  -f_{2}^e   & -f_{1}^e  & f_{4}^e  & f_{3}^e \\ f_{2}^e   & f_{1}^e  & -f_{4}^e  & -f_{3}^e\\ f_{2}^e   & f_{1}^e  & -f_{4}^e  & -f_{3}^e
\end{pmatrix}
$$
and $\Lambda_0=\diag(\frac{1}{f_1^e},\frac{1}{f_2^e},\frac{1}{f_3^e},\frac{1}{f_4^e})$ in Lemma \ref{L1} thanks to \cite[Lemma 1]{yang2025}.

According to the expressions of $\Lambda_1$ and $\Lambda_2$,  the incoming variables of the left, right, bottom and top edges of the boundary are $f_1(t,0,y),f_2(t,1,y),f_3(t,x,0),f_4(t,x,1)$, respectively, while the outgoing variables  from the left, right, bottom and top edges of the boundary are $f_2(t,0,y),f_1(t,1,y),f_4(t,x,0),f_3(t,x,1)$, respectively.

\subsection{Numerical boundary conditions}
In this subsection, we will give three different feasible numerical boundary conditions, local or non-local.

The boundary term $\mathcal{B}$ in Lemma \ref{L2} can be expressed as
$$
\begin{aligned}
	\mathcal{B}= &U\Delta x\left( \sum_{j_2=1}^{N-1}  (f_{1,(0,j_2)}^n)^2\left(\frac{\alpha}{f_1^e}+\exp(-U\Delta x)\right)+\sum_{j_1=1}^{N-1}  (f_{3,(j_1,0)}^n)^2\left(\frac{\alpha}{f_3^e}+\exp(-U\Delta x)\right) \right)\\
	&-U\Delta x\sum_{j_2=1}^{N-1}  (f_{1,(N-1,j_2)}^n)^2\left(\frac{\alpha}{f_1^e}+\exp(-U(N-1)\Delta x)\right) \\ &-U\Delta x\sum_{j_1=1}^{N-1}  (f_{3,(j_1,N-1)}^n)^2\left(\frac{\alpha}{f_3^e}+\exp(-U(N-1)\Delta x)\right)  \\
	&-U\Delta x\left( \sum_{j_2=1}^{N-1}  (f_{2,(1,j_2)}^n)^2\left(\frac{\alpha}{f_2^e}+\exp(U\Delta x)\right)+\sum_{j_1=1}^{N-1}  (f_{4,(j_1,1)}^n)^2\left(\frac{\alpha}{f_4^e}+\exp(U\Delta x)\right) \right)\\
	&+U\Delta x\left( \sum_{j_2=1}^{N-1}  (f_{2,(N,j_2)}^n)^2\left(\frac{\alpha}{f_2^e}+\exp(UN\Delta x)\right)+\sum_{j_1=1}^{N-1}  (f_{4,(j_1,N)}^n)^2\left(\frac{\alpha}{f_4^e}+\exp(UN\Delta x)\right) \right).
\end{aligned}
$$

First, we consider the trivial numerical boundary condition (\ref{3.3}): 
\begin{equation} \label{4.3}
	\begin{aligned}
		f_{1,(0,j_2)}^n&=0,\quad \forall j_2 \in \{1,\cdots,N-1\},\\
		f_{2,(N,j_2)}^n&=0,\quad  \forall j_2 \in \{1,\cdots,N-1\},\\
		f_{3,(j_1,0)}^n&=0,\quad  \forall j_1 \in \{1,\cdots,N-1\},\\
		f_{4,(j_1,N)}^n&=0,\quad  \forall j_1 \in \{1,\cdots,N-1\},
	\end{aligned}
\end{equation}
and clearly $\mathcal{B} \leq 0$ with (\ref{4.3}).

Now we consider more complicated numercial boundary conditions. We first choose the zero numerical boundary conditions on the left, right and top edges:
\begin{equation} \label{4.4}
	\begin{aligned}
		f_{1,(0,j_2)}^n&=0,\quad \forall j_2 \in \{1,\cdots,N-1\},\\
		f_{2,(N,j_2)}^n&=0,\quad  \forall j_2 \in \{1,\cdots,N-1\},\\
		f_{4,(j_1,N)}^n&=0,\quad  \forall j_1 \in \{1,\cdots,N-1\}.
	\end{aligned}
\end{equation}
For the bottom edge, we consider the following two  non-local numerical boundary conditions. 

The first one is to assign the value of the numerical incoming variables of the bottom edge in terms of the of the numerical outgoing variables of the left edge: 
\begin{equation} \label{4.5}
	f_{3,(j_1,0)}^n=kf_{2,(1,j_1)}^n,\quad  \forall j_1 \in \{1,\cdots,N-1\},
\end{equation}
where the tuning parameter $k$ serves as the control variable to be chosen.
With the numerical boundary conditions (\ref{4.4}) and (\ref{4.5}), we have
$$
\begin{aligned}
	\mathcal{B} &\leq U\Delta x \sum_{j_1=1}^{N-1} \left( (f_{3,(j_1,0)}^n)^2\left(\frac{\alpha}{f_3^e}+\exp(-U\Delta x)\right) 
	- (f_{2,(1,j_1)}^n)^2\left(\frac{\alpha}{f_2^e}+\exp(U\Delta x)\right) \right) \\
	&= U\Delta x \sum_{j_1=1}^{N-1} (f_{2,(1,j_1)}^n)^2\left(k^2\left(\frac{\alpha}{f_3^e}+\exp(-U\Delta x)\right) 
	-\left(\frac{\alpha}{f_2^e}+\exp(U\Delta x)\right) \right).
\end{aligned}
$$
Consequently, $\mathcal{B}\leq 0$ provided that
$$
|k| \leq \sqrt{\frac{\alpha/f_2^e+\exp(U\Delta x)}{\alpha/f_3^e+\exp(-U\Delta x)}},
$$
Note that $\Delta x \in (0,1)$, then
$$
|k| \leq \sqrt{\frac{\alpha/f_2^e+1}{\alpha/f_3^e+1}}
$$ 
is sufficient for $\mathcal{B} \leq0$, which is essential in the stabilization of the numerical schemes. 

Besides (\ref{4.5}), the numerical boundary conditions for the bottom edge can alternatively be specified as
\begin{equation} \label{4.6}
	f_{3,(j_1,0)}^n=k_1f_{2,(1,j_1)}^n+k_2f_{4,(j_1,1)}^n,\quad  \forall j_1 \in \{1,\cdots,N-1\},
\end{equation}
where $k_1$ and $k_2$ are our control variables.
With the numerical boundary conditions (\ref{4.4}) and (\ref{4.6}), we have
$$
\begin{aligned}
	\mathcal{B} \leq& U\Delta x \sum_{j_1=1}^{N-1} \left( (f_{3,(j_1,0)}^n)^2\left(\frac{\alpha}{f_3^e}+\exp(-U\Delta x)\right)- (f_{2,(1,j_1)}^n)^2\left(\frac{\alpha}{f_2^e}+\exp(U\Delta x)\right) \right) \\
	&- U\Delta x\sum_{j_1=1}^{N-1}  (f_{4,(j_1,1)}^n)^2\left(\frac{\alpha}{f_4^e}+\exp(U\Delta x)\right)\\
	\leq& U\Delta x \sum_{j_1=1}^{N-1} (f_{2,(1,j_1)}^n)^2 \left( 2k_1^2\left(\frac{\alpha}{f_3^e}+\exp(-U\Delta x)\right)- \left(\frac{\alpha}{f_2^e}+\exp(U\Delta x)\right) \right) \\
	&+U\Delta x \sum_{j_1=1}^{N-1} (f_{4,(j_1,1)}^n)^2 \left( 2k_2^2\left(\frac{\alpha}{f_3^e}+\exp(-U\Delta x)\right)- \left(\frac{\alpha}{f_4^e}+\exp(U\Delta x)\right) \right)
\end{aligned}
$$
Similar computations conclude that 
$$
|k_1| \leq \sqrt{\frac{\alpha/f_2^e+1}{2(\alpha/f_3^e+1)}}, \quad |k_2| \leq \sqrt{\frac{\alpha/f_4^e+1}{2(\alpha/f_3^e+1)}}
$$
are sufficient for $\mathcal{B}\leq0.$

\subsection{Numerical simulation}
In numerical simulations, we take  $U=1$,  $f_1^e=0.4$, $f_2^e=0.3$, $f_3^e=0.2, f_4^e=0.6$.  Clearly, $f_1^ef_2^e=f_3^ef_4^e$ holds, which implies that they are uniform steady states.

\textbf{Simulation I: Explicit schemes with different spatial grids.} We first assume $\sigma=1$ and apply the numerical schemes (\ref{3.1}) and (\ref{3.2}), and the initial condition is set to be $(1,1,1,1)$. According to \cite[Lemma 1]{yang2025}, the matrices $\Lambda_0,\Lambda$ and $P$ can be explicitly computed. Consequently, all the constants $M,m,\lambda_m,\lambda_M,\lambda,\mu,\| P\|,\| Q\|$ can be explicitly computed. 
Once the spatial step $\Delta x$ is fixed, we can use these constants to determine the corresponding $\Delta t$ as required in Theorem \ref{T1}. For the numerical boundary condition (\ref{4.3}) and different spatial step sizes $\Delta x$, the time evolution of the logarithm of the $\ell^2$-norm is presented in Figure \ref{fig3}. The results show that the numerical solution of the linearized system (\ref{4.2}), with the prescribed initial data, converges exponentially to the origin. Moreover, the decay rates corresponding to different spatial discretizations are nearly identical, indicating that the $\ell^2$-norm of the solutions decays uniformly, independent of both $\Delta x$ and $\Delta t$, as stated in Theorem \ref{T1}.
\begin{figure}[htbp]
	\centering
	\includegraphics[width=0.7\textwidth]{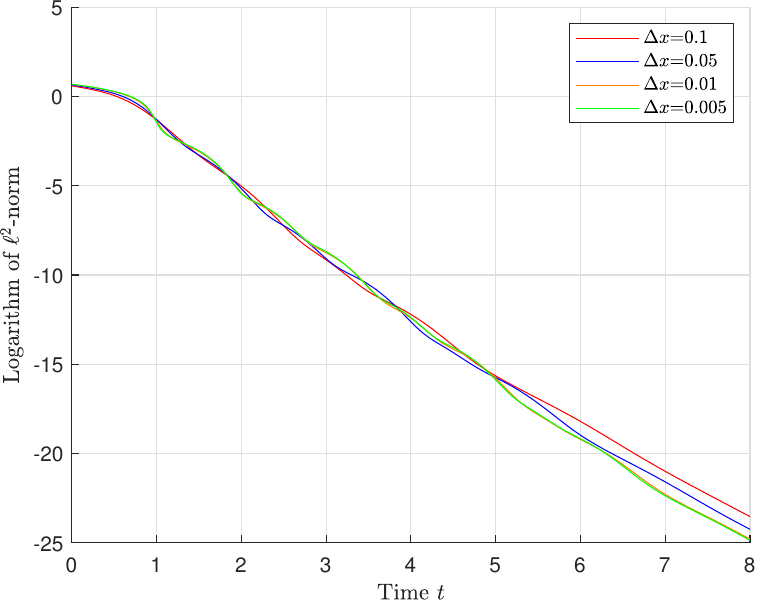}
	\caption{Simulation I, Time evolution of logarithm of $\ell^2$-norm of the solution with four different  spatial grids.}
	\label{fig3}
\end{figure}

\textbf{Simulation II: Explicit schemes with different numerical boundary conditions.}
Besides, given $\Delta x = 0.05$, $\Delta t=0.01$ and $\sigma=1,$ we also use the upwind scheme (\ref{3.1}) and the explicit scheme (\ref{3.2}). The numerical results corresponding to the numerical boundary conditions (\ref{4.3}), (\ref{4.4})--(\ref{4.5}), and (\ref{4.4})--(\ref{4.6}) are presented in Figure \ref{fig4}, where 
$k=1$ in (\ref{4.5}) and $k_1=k_2=1$ in (\ref{4.6}). These results demonstrate that the two non-local numerical boundary conditions (\ref{4.4})--(\ref{4.5}) and (\ref{4.4})--(\ref{4.6}) are feasible, while the trivial numerical boundary condition (\ref{4.3}) yields the strongest damping effect, as stated in Remark \ref{R2}.
\begin{figure}[htbp]
	\centering
	\includegraphics[width=0.7\textwidth]{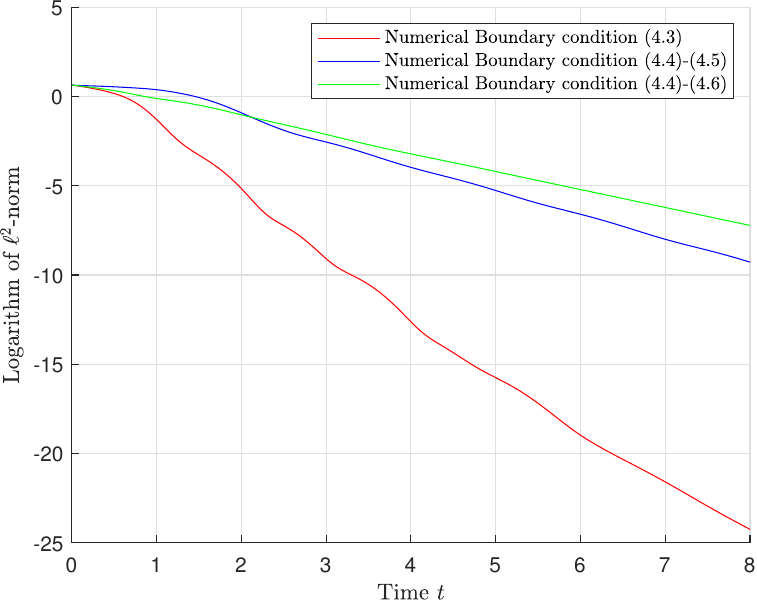}
	\caption{Simulation II, Time evolution of logarithm of $\ell^2$-norm of the solution with three different numercial boundary conditions.}
	\label{fig4}
\end{figure}

\textbf{Simulation III: Semi-implicit schemes with different mean free path $\sigma$.}
Finally, we set $\Delta x =0.1$, $\Delta t = 0.05$ and impose the numerical boundary condition (\ref{4.3}). For different values of mean free path $\sigma$, we apply the upwind scheme (\ref{3.1}) in combination with the implicit scheme (\ref{3.8}).  The results for different values of $\sigma$ are shown in Figure \ref{fig5}. As expected, a smaller $\sigma$ does not require a smaller time step $\Delta t$, but it results in a slower decay rate, as discussed in Remark~\ref{R3}.

For comparison, we also present the result (the black dashed line) obtained with $\sigma = 0.02$ when using the upwind scheme (\ref{3.1}) together with the explicit scheme (\ref{3.2}) under the same discretization parameters and numerical boundary condition. The divergent behavior observed in this case  clearly indicates that the explicit scheme becomes unstable in the presence of the stiff source term.
\begin{figure}[htbp]
	\centering
	\includegraphics[width=0.7\textwidth]{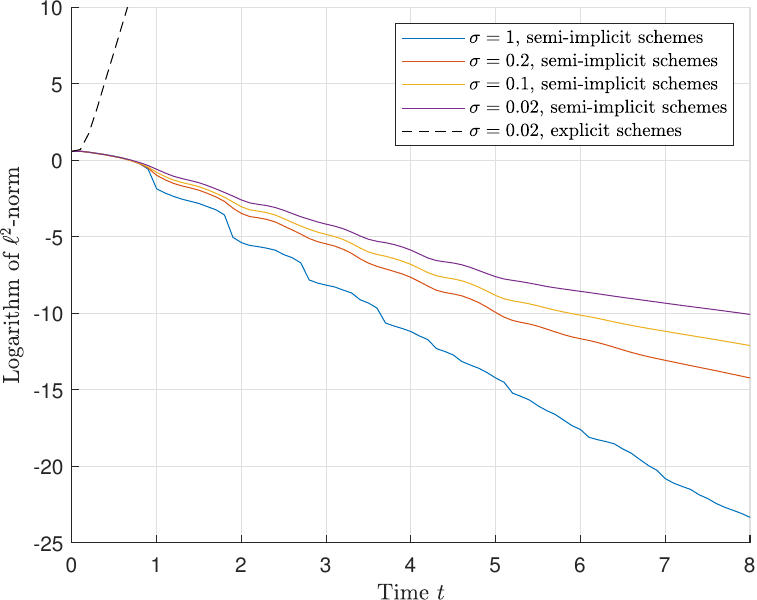}
	\caption{Simulation III, Time evolution of logarithm of $\ell^2$-norm of different $\sigma$ with implicit or explicit schemes.}
	\label{fig5}
\end{figure}

\section{Concluding remarks}
In this paper, we extend our recent results on multi-dimensional discrete-velocity models to the numerical level. By adopting an operator splitting scheme and introducing a suitable discrete Lyapunov function, we derive numerical control laws that ensure the corresponding numerical solutions decay exponentially in time. To handle stiff source terms, we also use an implicit scheme for the collision part and prove the stability of the resulting schemes. As an application, numerical control laws (\ref{4.3}), (\ref{4.4})–(\ref{4.5}), and (\ref{4.4})–(\ref{4.6}) are designed for the 2-D coplanar model. The theoretical results are validated through three types of numerical simulations.

We only consider the domain $\Omega=(0,1)^d$, and it is trivial to extend the results to the  general hyperrectangular domains. However, the extension to general domains is an   intricate but interesting question. In those cases, the finite difference schemes (\ref{3.1})-(\ref{3.2}) or (\ref{3.1})-(\ref{3.8}) should be replaced by  a finite volume discretization.
It is also interesting to extend our numerical stability results to the semi-linear discrete-velocity models. These issues are our ongoing work.

\bibliographystyle{plain}
\bibliography{ref}
\end{document}